\documentclass[10pt]{article}
\font\smallit=cmti10

\usepackage{amssymb,latexsym,amsmath,epsfig,amsthm} 
\usepackage[dvipdfmx,svgnames]{xcolor}
\usepackage{tikz}

\makeatletter

\renewcommand\section{\@startsection {section}{1}{\z@}
{-30pt \@plus -1ex \@minus -.2ex}
{2.3ex \@plus.2ex}
{\normalfont\normalsize\bfseries\boldmath}}

\renewcommand\subsection{\@startsection{subsection}{2}{\z@}
{-3.25ex\@plus -1ex \@minus -.2ex}
{1.5ex \@plus .2ex}
{\normalfont\normalsize\bfseries\boldmath}}

\renewcommand{\@seccntformat}[1]{\csname the#1\endcsname. }

\newcommand\Pp{${\mathcal P}$-position }
\newcommand\Pps{${\mathcal P}$-positions }
\newcommand\Np{${\mathcal N}$-position }

\newcommand\Z{{\mathbb Z}}
\newcommand\Min{{\mathrm{Min}}}
\newcommand\Max{{\mathrm{Max}}}
\newcommand\OO{{\mathbb O}}

\makeatother

\newtheorem{theorem}{Theorem}

\newtheorem{proposition}{Proposition}
\newtheorem{corollary}{Corollary}

\theoremstyle{definition}
\newtheorem{definition}{Definition}

\newtheorem{remark}{Remark}

\newtheorem{question}{Question}

\begin{document}

\begin{center}
\uppercase{\bf On variations of Yama Nim and Triangular Nim}
\vskip 20pt
{\bf Shun-ichi Kimura\footnote{partially supported by JSPS Kakenhi 20K03514, 23K03071, 20H01798}}\\
{\smallit Department of Mathematics, Hiroshima University, Hiroshima Japan}\\
{\tt skimura@hiroshima-u.ac.jp}

\vskip 10pt

{\bf Takahiro Yamashita\footnote{}}\\
{\smallit Department of Mathematics, Hiroshima University, Hiroshima Japan}\\
{\tt d236676@hiroshima-u.ac.jp}
\end{center}
\vskip 20pt
\vskip 30pt

\centerline{\bf Abstract}
\noindent

Yama Nim is a two heaps Nim game introduced in the second author's Master Thesis, where the player  takes more than $2$ tokens from one heap, and return $1$ token to the other heap.  Triangular Nim is a generalization, where the player takes several tokens from one heap, and return some tokens (at least one token)  to the other heap, so that the total number of the tokens in the heaps decrease strictly.  In this paper, we investigate their winning strategies, Grundy numbers, and their variations and generalizations.  Particularly interesting is the Wythoff variations, where in addition to the Yama/Triangular Nim moves, the player is allowed to take tokens from both heaps, say $i$ tokens from the first heap and $j$  tokens from the other, under some restrictions for $i$ and $j$.  For example when we force $i=j>0$ for the Triangular Nim, then $(x, y)\in \Z_{\ge 0}$ with $x\le y$ is in the \Pp if and only if $(x, y)\in \{(0, 0), (0, 1), (1, 3), (3, 6), (6, 10), (10, 15), \ldots\}$, namely  the winning strategy is described by triangular numbers.  In other rulesets, we also found examples where the square numbers, pentagonal numbers, geometric progressions, and so on.

\pagestyle{myheadings} 
\thispagestyle{empty} 
\baselineskip=12.875pt 
\vskip 30pt


\section{Preliminary}

In this paper, we treat a zero-sum combinatorial game played by two players, say Left and Right.
Unless otherwise stated, we consider impartial short games \cite[pp 34-36]{S} as follows:

\begin{definition} The ruleset $\Gamma$ with the set of positions $X$ is called impartial short game if it satisfies the following three conditions:
\begin{enumerate}
\item[(1)] (Finite) Each $G\in X$ has just finitely many distinct subpositions.
\item[(2)] (Loopfree) Every run of $G\in X$ has finite length.
\item[(3)] (Impartial) Both players have exactly the same moves available from each $G\in X$.
\end{enumerate}
\end{definition}

For the most part, we work with Nim games with two heaps, and then our set of positions $X$ will be written as $$ X=\Z_{\ge 0}\times \Z_{\ge 0}=\{(x, y)\in \Z\times \Z\, | \, x\ge 0, y\ge 0\}$$
where $x$ is the number of tokens in the first heap and $y$ in the second.  Our ruleset can be described by the option function $f: X\to {\mathrm Pow}(X)$, where ${\mathrm Pow}(X)$ the set of subsets of $X$, so that from the position $m\in X$, the player can choose her/his move from $f(m)$.  We also assume normal play for the most part of this paper, namely the last player wins.  A position $m\in X$ is a \Pp if the previous player has the winning strategy, and $m$ is in \Np if the next player has the winning strategy.  We can also define Grundy numbers to give more information \cite[Chapter IV]{S}, so that $m\in X$ is in \Pp if and only if its Grundy number is $0$.  

It is well-known that a subset $S\in X$ coincides with the set of \Pps in the normal play if and only if 
\begin{enumerate}
\item[(1)] For any $m\in S$ and for any its option $m'\in f(m)$, we have $m'\not\in S$.
\item[(2)] For any $m\in X\backslash S$, there is some option $m'\in f(m)$ such that $m'\in S$.
\end{enumerate}

We start by generalizing the statement above to Grundy numbers.

\begin{proposition}\label{grundy} Let  $\Gamma$ be the ruleset and $X$ its set of positions.  Assume that $X$ is divided into disjoint subsets $X=S_0 \sqcup S_1 \sqcup \cdots$, indexed by non-negative integers, in such a way that for any $m\in S_i$, its option $m'\in f(m)$ is not $S_i$, and for any $m\in S_i$ with $j<i$, there is some option $m'\in f(m)$ such that $m'\in S_j$.  Then $S_i$ is exactly the set of positions with its Grundy numbers equal to $i$.
\end{proposition}

\begin{proof} Let $\Z_{\ge 0}$ be the $1$-heap Nim (and its positions by abuse of notation) and consider the disjunctive sum $\Gamma+\Z_{\ge 0}$.  It is enough to show that $S=\{(m, i)\in X\times \Z_{\ge 0} \, |\, m\in S_i\}$ is the set of ${\mathcal P}$-positions.  When $(m, i)\in S$, then the options are $(m', i)$ with $m'\in f(m)$, or $(m, j)$ with $j<i$, and both of them are clearly out of $S$.  On the other hand, if $(m, k)\in X\times \Z_{\ge 0}$ with $m\in S_i$, then either $i>k$ or $i<k$.  When $i>k$, the player has the option $m'\in m$ with $m'\in S_k$, and when $i<k$, the player can choose the option $(m, k)\to (m, i)$.
\end{proof}

The following is a subset of the results by Yohei Yamasaki \cite{Y1} and \cite{Y2}, stated in a way convenient for our purpose.

\begin{definition} Let $\Gamma$ be a ruleset with $X$ the set of positions.  For each $m\in X$ with its Grundy number for the normal play $g(m)$, we determine either it is {\em in the final phase} or {\em normal phase}, inductively as follows:
\begin{enumerate} 
\item[(1)] If $g(m)=0$ and $m$ has no options, namely $f(m)=\emptyset$, then $m$ is in the final phase.
\item[(2)] If $g(m)=1$ and for all $m'\in f(m)$ with $g(m')=0$, $m'$ is in the final phase, then $m$ is in the final phase.  Otherwise, $m$ is in the normal phase.
\item[(3)] If $g(m)=0$ with $f(m)\not=\emptyset$, $m$ is in the final phase if and only if there exists an option $m'\in f(m)$ with $g(m')=1$ in the final phase.
\item[(4)] If $g(m)>1$, then $m$ is in the normal phase.
\end{enumerate}
We say $\Gamma$ is {\em admissible} if for each $m\in X$ with $g(m)>1$, either there is an option $m'\in f(m)$ with $g(m')=0$  in the normal phase, or $m"\in f(m)$ with $g(m")=1$ in the final phase.

Moreover, we say that $\Gamma$ is {\em universally admissible} if it satisfies the following $3$ conditions:
\begin{enumerate}
\item[(i)] If $m\in X$ has Grundy number $1$ and in the normal phase, for any option $m'\in f(m)$ with $g(m)=0$, $m$ is in the normal phase.
\item[(ii)] If $m\in X$ has Grundy number $0$ and in the final phase, for any option $m'\in f(m)$ with $g(m)=1$, $m$ is in the final phase.
\item[(iii)] If $m\in X$ has Grundy number $g(m)>1$, then it has options $m'_0$ and $m'_1$ so that $g(m'_0)=0, g(m'_1)=1$, and either both of them are in the normal phase, or both of them are in the final phase.
\end{enumerate}
\end{definition}

\begin{remark}\label{misere1} Universally admissible rulesets are admissible.  \end{remark} 

\begin{remark}\label{misere2} Assume that for any position $m\in X$ with $g(m)=0$, there exists some option $m'\in f(m)$ with $g(m')=1$.  Then Conway induction implies that for any $m\in X$ with $g(m)\le 1$, $m$ is in the final phase.  If moreover the ruleset is admissible, it is universally admissible. \end{remark}

\begin{proposition}\label{misereP} When the ruleset $\Gamma$ is admissible, its winning strategy for the mis\`ere play can be described as follows:  The \Pps for the mis\`ere play for $\Gamma$ is 
$$\{m\in X\, |\, \hbox{$g(m)=0$ in the normal phase}\}\cup \{m\in X\, |\, \hbox{$g(m)=1$ in the final phase}\}$$

The disjoint sum of universally admissible rulesets $\Gamma_1+\Gamma_2$ is again universally admissible, with $(m_1, m_2)$ is in the final phase if and only if  both  $m_1$ and $m_2$ are in the final phase.
\end{proposition}

\section{Grundy  numbers for Yama Nim and Triangular Nim}

In this section, we study the Grundy numbers for Yama Nim and Triangular Nim.

First, we start by the ruleset of Yama Nim, which was introduced in the Master Thesis of the second author \cite{Y23}, where also its Grundy numbers and winning strategies both for the normal play and the mis\`ere play are given (see Theorem \ref{yama1} and Corollay \ref{yama2} below).

\begin{definition} Let $X=\Z_{\ge 0}\times \Z_{\ge 0}$ be the set of positions of Yama Nim.  The options are given by  $$f((x, y))=\{(x-i, y+1)\, |\, 2\le i\le x\} \cup \{(x+1, y-i)\, |\,  2\le i \le y\}.$$
\end{definition}

\begin{theorem}\label{yama1} The Grundy number for the normal play of Yama Nim is given by
$$g((x,y))=\left\{\begin{array}{cc} 0& (|x-y|\le 1)\\ \Min(x, y)+1 & (|x-y|>1)\end{array}.\right.$$
\end{theorem}

\begin{proof} We define $S_0=\{(x, y)\, |\, |x-y|\le 1\}$ and for $i>0$, $S_i=\{(x, y)\not\in S_0\, |\, \Min(x, y)+1=i\}$.  We check the conditions for Proposition \ref{grundy}. First, we start showing that if $(x, y)\in S_i$, then its option $(x', y')$ is not in $S_i$.

When $(x, y)\in S_0$, by symmetry, it is enough to show that $(x-i, y+1)\not\in S_0$ for $i\ge 2$, which follows from $$|(x-i)-(y+1)|=|(x-y)-(i+1)|\ge |i+1|-|x-y|\ge 2.$$  When $(x, y)\in S_i$ with $i>0$, to show that its options are not in $S_i$, we may assume that $x<y$ with $y-x>1$.  If $(x, y)\to (x', y')$ then either $x'<x$ (in which case $\Min(x', y')<\Min(x, y)$ and $(x', y')$ cannot be in $S_i$), or $x'=x+1$ and $y'\le y-2$, in which case if $\Min(x', y')=\Min(x, y)$, then we need to have $y'=x$, then $(x', y')\in S_0$, not in $S_i$.  

We also need to show that if $(x, y)\in S_i$ with $i>0$ and $j<i$, then there is some option $(x', y')\in S_j$ of $(x, y)$.  We may assume $x<y$ with $y-x\ge 2$.  When $j=0$, we can move to $(x+1, x) \in S_0$.  When $j>0$, we can move to $(x+1, j-1)$ because $j\le i-1 \le x\le y-2$.  
\end{proof}

\begin{corollary}\label{yama2} The position $(x, y)$ has no option (therefore its Grundy number is $0$), if and only if $x\le 1$ and $y\le 1$.  When $(x, y)$ has some option with its Grundy number $0$, then there is an option $(x', y')$ with the Grundy number $1$, so in the mis\`ere play, the \Pps are $(x, y)$ with $\Min(x, y)=0$ and $\Max(x, y)\ge 2$.  \end{corollary}

\begin{proof} When $g(x, y)=0$ with non-empty options (therefore $\Max(x, y)\ge 2$), say $x\le y$ with $y-x\le 1$ and $y\ge 2$, then the player has the option $(x+1, 0)$, which has Grundy number $1$.  We can apply Remark \ref{misere2}.  \end{proof}

The \Pps of Triangular Nim is exactly  same as Yama Nim (hence the same outcome), but Grundy numbers behave in a more interesting way.  In particular, the winning strategy for the mis\`ere play, Triangular Nim is different from the mis\`ere play of Yama Nim.  This ruleset was born when the second author talked with Professor Urban Larsson, Professor Indrajit Saha, Professor Koki Suetsugu, Professor Tomoaki Abuku, and Professor Hironori Kiya about Yama Nim on the occasion of a conference at Kyushu University.

\begin{definition}\label{triangular} 
Let $X=\Z_{\ge 0}\times \Z_{\ge 0}$ be the set of positions of Triangular Nim.  The options are given by  $$f((x, y))=\{(x-i, y+j)\, |\, 2\le i\le x, 1\le j<i\} \cup \{(x+j, y-i)\, |\,  2\le i \le y, 1\le j<i\}.$$
\end{definition}

\begin{theorem}\label{tri1} The Grundy number for the normal play of Triangular Nim is given as follows:
$$g((x, y))=\left\{\begin{array}{cc} 0& (|x-y|\le 1)\\ \dfrac{d(d-1)}{2}+k & \left(\begin{array}{l} |y-x|=:d>1\\y+x\ge d^2+2\\ \hbox{$k<d$ and $x+y-(d^2+2)\equiv 2k \mod 2d$}\end{array}\right)\\
x+y-1& \hbox{($|x-y|=:d>1$ and $x+y\le d^2+1$)}\end{array}\right..$$
as in the table below:

\vskip 5mm

$\begin{array}{c||c|c|c|c|c|c|c|c|c|c|c|c|c|c|c|c}
x\backslash y&0&1&2&3&4&5&6&7&8&9&10&11&12&13&14&15\\ \hline \hline 
 0&0 & 0 & 1 & 2 & 3 & 4 & 5 & 6 & 7 & 8 & 9 & 10 & 11 & 12 & 13 & 14 \\  \hline 
 1&0 & 0 & 0 & 3 & 4 & 5 & 6 & 7 & 8 & 9 & 10 & 11 & 12 & 13 & 14 & 15 \\ \hline 
 2&1 & 0 & 0 & 0 & 1 & 6 & 7 & 8 & 9 & 10 & 11 & 12 & 13 & 14 & 15 & 16 \\ \hline 
 3&2 & 3 & 0 & 0 & 0 & 2 & 8 & 9 & 10 & 11 & 12 & 13 & 14 & 15 & 16 & 17 \\ \hline 
 4&3 & 4 & 1 & 0 & 0 & 0 & 1 & 3 & 11 & 12 & 13 & 14 & 15 & 16 & 17 & 18 \\ \hline 
 5&4 & 5 & 6 & 2 & 0 & 0 & 0 & 2 & 4 & 13 & 14 & 15 & 16 & 17 & 18 & 19 \\ \hline 
 6&5 & 6 & 7 & 8 & 1 & 0 & 0 & 0 & 1 & 5 & 15 & 16 & 17 & 18 & 19 & 20 \\ \hline 
 7&6 & 7 & 8 & 9 & 3 & 2 & 0 & 0 & 0 & 2 & 3 & 6 & 18 & 19 & 20 & 21 \\ \hline 
 8&7 & 8 & 9 & 10 & 11 & 4 & 1 & 0 & 0 & 0 & 1 & 4 & 7 & 20 & 21 & 22 \\ \hline 
 9&8 & 9 & 10 & 11 & 12 & 13 & 5 & 2 & 0 & 0 & 0 & 2 & 5 & 8 & 22 & 23 \\ \hline 
10& 9 & 10 & 11 & 12 & 13 & 14 & 15 & 3 & 1 & 0 & 0 & 0 & 1 & 3 & 9 & 24 \\ \hline 
 11&10 & 11 & 12 & 13 & 14 & 15 & 16 & 6 & 4 & 2 & 0 & 0 & 0 & 2 & 4 & 6 \\ \hline 
 12&11 & 12 & 13 & 14 & 15 & 16 & 17 & 18 & 7 & 5 & 1 & 0 & 0 & 0 & 1 & 5 \\ \hline 
13& 12 & 13 & 14 & 15 & 16 & 17 & 18 & 19 & 20 & 8 & 3 & 2 & 0 & 0 & 0 & 2 \\ \hline 
 14&13 & 14 & 15 & 16 & 17 & 18 & 19 & 20 & 21 & 22 & 9 & 4 & 1 & 0 & 0 & 0 \\ \hline 
 15&14 & 15 & 16 & 17 & 18 & 19 & 20 & 21 & 22 & 23 & 24 & 6 & 5 & 2 & 0 & 0 \\ \hline 
\end{array}$

\end{theorem}

\vskip 5mm

One can restate Theorem \ref{tri1} in an interesting way.  We need some preparation for the statement.

\begin{definition}  For each $g\in \Z_{\ge 0}$, take $d\in \Z_{>0}$ as the natural number which satisfies the condition $$\dfrac{d(d-1)}{2}\le g<\dfrac{d(d+1)}{2},$$ namely the $d$-th triangular number $d(d+1)/2$ is the first triangular number larger than $g$.  Take the arithmetic progression $a_{g, k}=1+g+kd$, whose first term $a_{g,0}=1+g$ with the common difference $d$.  For example for $g=0$, the arithmetic progression is  $$\{a_{0,0}, a_{0, 1}, a_{0, 2}, a_{0, 3}, \ldots\}=\{1, 2, 3, 4, \ldots\},$$ and for $g=1$, $$\{a_{1,0}, a_{1, 1}, a_{1, 2} a_{1, 3}, \ldots\}=\{2, 4, 6, 8, \ldots\},$$ and for $g=2$, we have $$\{a_{2, 0}, a_{2, 1}, a_{2, 2}, a_{2, 3}, \ldots\}=\{3, 5, 7, 9, \ldots\},$$ and for $g=3$, we have $$\{a_{3, 0}, a_{3, 1}, a_{3, 2}, a_{3,3}, \ldots\}=\{4, 7, 10, 13, \ldots\}.$$\end{definition}

\begin{theorem}\label{tri2} In Triangular Nim, if $x=y$, the Grundy number for $(x, y)$ is $0$.  When $(x, y)$ is a position with $x<y$, 
 the Grundy number  $g(x, y)$ for $(x, y)$ is given by 
$$g(x, y)=\left\{\begin{array}{cl} g& \hbox{(when $(x, y)=(a_{g, k}, a_{g, k+1})$ for some $g\ge 0$ and $k\ge 0$)}\\
x+y-1&\hbox{(otherwise)}\end{array}\right..$$
\end{theorem}

For example, for $(x, y)$ with $x<y$, we have \begin{eqnarray*} g(x, y)=0 &\iff& (x, y)=(0, 1), (1, 2), (2, 3), (3, 4), (4, 5),\ldots, \\ g(x, y)=1 &\iff& (x, y)=(0, 2), (2, 4), (4, 6), (6, 8), (8, 10), \ldots,\\ g(x, y)=2 &\iff& (x, y)=(0, 3), (3, 5), (5, 7), (7, 9), (9, 11), \ldots, \\g(x, y)=3 &\iff& (x, y)=(1, 3), (0, 4), (4, 7), (7, 10), (10, 13), (13, 16), \ldots, \\ g(x, y)=4 &\iff& (x, y)=(1, 4), (0, 5), (5, 8), (8, 11), (11, 14), (14, 17), \ldots. \end{eqnarray*}
Notice that except for the first few exceptions (which corresponds to the ``otherwise" cases), the arithmetic progressions appear to describe the pattern.

\begin{proof}  The equivalence of Theorem \ref{tri1} and Theorem \ref{tri2} is easy to check and left to the reader.  

We define
\begin{eqnarray*} S_0&=&\{(x, y)\, |\, |x-y|\le 1\}\\
S_1&=&\{(0, 2), (2, 4), (4, 6), (6, 8), \ldots\}\cup \{\ldots, (8, 6), (6, 4), (4, 2), (2,0)\}\\
S_2&=&\{(0, 3), (3, 0)\}\cup \{(3, 5), (5, 7), (7, 9), (9, 11), \ldots\}\cup \{\ldots, (11, 9), (9, 7), (7, 5), (5, 3)\}\\
S_3&=&\{(	1, 3), (0, 4)\}\cup \{(4, 7), (7, 10), (10, 13), (13, 16), \ldots\}\cup \{\ldots, (11, 9), (9, 7), (7, 5), (5, 3)\}\end{eqnarray*} and for general $k>0$, we define $S_g$ to be the union of the three sets $T_g\cup U_g \cup V_g$ where
\begin{eqnarray*}T_g&=&\{(x, y)\, |\, \hbox{$x+y=g+1$ and $(x+y)\le (x-y)^2+1$}\}\\ U_g&=& \{(1+g, 1+g+d), (1+g+d, 1+g+2d), (1+g+2d, 1+g+3d), \ldots\}\\ V_g&=&\{\ldots, (1+g+3d, 1+g+2d), (1+g+2d, 1+g+d), (1+g+d, 1+g)\}.\end{eqnarray*}

We will check the conditions for Proposition \ref{grundy}.  First, we show that when $(x, y)\in S_0$ with its option $(x, y)\to (x', y')$, then $(x', y')\not\in S_0$.  This is similar to the Yama-Nim case.  We have $|x-y|\le 1$, and we may assume that $(x', y')=(x+i y-j)$ where $i\ge 2$ and $i>j\ge 1$.  Then $$|x'-y'|=|(x-y)+(i+j)|\ge |i+j|-|x-y|\ge 3-1=2$$
hence $(x', y')\not \in S_0$.  When $(x, y)\in S_g$, there are two cases, either $(x, y)\in T_g$ or $(x, y)\in U_g\cup V_g$.  When $(x, y)\in T_g$, then we have $x\le g+1$ and $y\le g+1$.  When  $(x, y)\in U_g\cup V_g$, we have $x\ge g+1$ and $y\ge g+1$.  When $(x', y')$ is an option of $(x, y)$, then one component strictly increases, and the other decreases, so it is impossible to move between $T_g$ and $U_g\cup V_g$.  If both $(x, y)$ and $(x', y')$ are in $T_g$, we need to have $x'+y'<x+y$, but we have $x'+y'=g+1=x+y$, so it is impossible.  Finally we consider the case $(x, y)\in U_g\cup V_g$.  By symmetry, we may assume that $x<y$.  Let us write $a_k=1+g+kd, (k\in \{0, 1, 2, 3, \ldots)$ with $d(d-1)/2\le g<d(d+1)/2$, then we can write $x=a_k$ and $y=a_{k+1}$ for some $k$.  As $(x', y')\in U_g\cup V_g$, we need to have $x'=a_{\ell}$ and $y'=a_m$ for some $\ell$ and $m$, with $|\ell-m|=1$.  If $x'<x$ and $y'>y$, then we have $\ell<k<k+1<m$ and the relation $|\ell-m|=1$ is impossible.  If $x<x'$ and $y'<y$, then we have $k<\ell$ and $m<k+1$, then the only possibility for $|\ell-m|=1$ is the $\ell=k+1$ and $m=k$.  But then $(x', y')=(a_{k+1}, a_k)$, and $x+y=x'+y'$, which is not allowed in Triangular Nim.  We conclude that if $(x, y)\in S_g$ and $(x', y')$ is its option, then $(x', y')\not\in S_g$.

Now assume that $(x, y)\in S_g$ with $g>0$ and $h<g$, then we need to show that there is some option $(x', y')\in S_h$.  In Yama Nim, we have shown that each $(x, y)\not\in S_0$ has an option $(x', y')\in S_0$, and in Triangular Nim, we have more options than Yama Nim, so we may assume that $h>0$.
We may also assume that $x<y$, or to be more precise, $x+1<y$, as $(x, y)\not\in S_0$.  When $x\le h$, then its option contains $(h+1, 0)\in T_h\subset S_h$.  So we may assume $x\ge h+1$.  Let us consider the arithmetic progression $b_k=1+h+ke$ with $e(e-1)/2\le h<e(e+1)/2, k\in \{0, 1, 2, 3, \ldots\}$.  As $x\ge h+1=b_0$, there is some index $k$ such that $b_k\le x<b_{k+1}$.  If $x+y>b_k+b_{k+1}$, then $(b_{k+1}, b_k)\in V_h$ is an option of $(x, y)$, and we are done.  For example if $(x, y)\in U_g$, then it certainly happens, because $h<g$, we have $e\le d$ (where $d(d-1)/2\le g<d(d+1)/2$ is the number used to define $U_g$, so that $x=1+g+\ell d$ and $y=1+g+(\ell+1)d$  for some $\ell$), then $y-x\ge d\ge e$.  If $x+y\le b_k+b_{k+1}$, then $y=b_{k+1}-(x-b_k)\le e$ with the equality occurs if and only if $y=b_{k+1}$ and $x=b_k$, but it implies $(x, y)\in U_h\subset S_h$, contradicting to our assumption that $(x, y)\in S_g$ with $g>h$. 

 Now it is enough to show that if $x+y\le b_k+b_{k+1}$, then $(y-x)^2+2\le x+y$ holds, hence $(x, y)\not\in T_g$.  Assume, to the contrary, that $(y-x)^2+1\ge x+y$, then as $y=x+d$, we have $d^2+1\ge 2x+d$, and hence $x\le (d^2-d+1)/2$, which implies $x\le d(d-1)/2$, because $d^2-d+1$ is an odd number.  Then combined with $x\ge h+1>e(e-1)/2$, we obtain $$e(e-1)/2<x\le d(d-1)/2,$$  hence $e<d$.  On the other hand, our assumption $x+y\le b_k+b_{k+1}$ implies $y\le b_{k+1}-(x-b_k)\le b_{k+1}$, and hence $$b_k\le x<y\le b_{k+1}=b_k+e,$$ which implies $d=y-x\le e$, a contradiction.  
 We are done.
\end{proof}

\begin{corollary} The position $(x, y)$ is \Pp of Triangular Nim in the mis\`ere play, if and only if either $\Min(x, y)\ge 2$ and  $|x-y|\le 1$, or $(x, y)\in \{(0, 2), (2, 0)\}$. Triangular Nim is admissible, but is not universally admissible. \end{corollary}

\begin{proof} First, we observe that $(x, y)$ is in the final phase if and only if $(x, y)\in \{(0, 0), (1, 0), (0, 1), (2, 0), (1 ,1), (0, 2), (2, 1), (1, 2)\}$.  For the admissibility, we take a position $(x, y)$ with Grundy number larger than $1$, where we may assume $x<y$ with $y\ge x+2$, and we need to show that either $(x', y')$ with Grundy number $1$, in the final stage, or Grundy number $0$ in the normal stage lies in the option of $(x, y)$.  From $(x, y)$ with $y\ge 2$, $(y+1, y)$ is in the normal phase with Grundy number $0$.  Also from $(x, y)$ with $x\le 1$, being Grundy number larger than $1$ implies $y\ge 3$, and hence $(2, 0)$ lies in the option, which is Grundy number $1$ in the final stage.

Now we can apply  Proposition \ref{misereP} for the ${\mathcal P}$-positions.  From $(1,3)$, the options are $\{(2, 1), (2, 0), (3, 0)\}$.  $(2, 1)$ is in the final phase with Grundy number $0$, $(2, 1)$ is in the final phase with Grundy number $1$, and $(3, 0)$ has Grundy number $2$, so it is not universally admissible. \end{proof}

\begin{question} Koki Suetsugu pointed out a similarity between Grundy number of Triangular Nim and Grundy number of Moore Nim \cite[Fig. 1]{JM}.  Is there any deep reason for this phenomenon? \end{question}

\section{$ab$-Yama Nim and  $ab$-Triangular Nim}

In Yama Nim, the player takes at least $2$ tokens from one heap, and return exactly $1$ token to the other.  In Triangular Nim, the player takes at least $2$ tokens from one heap, and return at least $1$ token to the other so that the number of the tokens in the heaps decreases.  By changing $2$ to $a$ and $1$ to $b$, we get  generalizations of  these Nim games as follows:

\begin{definition}  Let $a$ be a positive integer and $b$ an integer with $a>b$.  Let $X=\Z_{\ge 0}\times \Z_{\ge 0}$ be the set of positions of $ab$-Yama Nim and $ab$-Triangular Nim.  The options for the $ab$-Yama Nim are given by  $$f((x, y))=\{(x-i, y+b)\, |\, a\le i\le x\} \cup \{(x+b, y-i)\, |\,  a\le i \le y\}.$$
Also the options for the $ab$-Triangular Nim are given by 
$$f((x, y))=\{(x-i, y+j)\, |\, a\le i\le x, b\le j<i\} \cup \{(x+b, y-i)\, |\,  a\le i \le y, b\le j<i\}.$$
\end{definition}

To describe their ${\mathcal P}$-positions, we need the following notation.

\begin{definition} Let $a$ be a positive integer and $b$ an integer with $a>b$.  For a non-negative integer $k\in \Z_{\ge 0}$, we define $R_k(a, b)$ to be $$R_k(a, b)=\{(x, y)\in \Z_{\ge 0}^2\, |\, k(a-b)\le x, y<k(a-b)+a\}.$$  When confusion is unlikely, we simply write $R_k$ for $R_k(a, b)$. \end{definition}

\begin{theorem}\label{abyama} The set of  \Pps of  $ab$-Triangular Nim is $\displaystyle \bigcup_{k=0}^{\infty} R_k(a, b)$.  When $b\ge 0$, the set of \Pps of $ab$-Yama Nim is the same: $\displaystyle \bigcup_{k=0}^{\infty} R_k(a, b)$. \end{theorem}

\begin{proof}  Take a position $(x, y)\not\in \bigcup R_k$ with $x\le y$.  We would like to show that $(x ,y)$ has some option $(x', y')$ in some $R_k$.  First, we assume that $b\ge 0$, and let $k=\left[\dfrac{x}{a-b}\right]$, then as $k(a-b)\le x<(k+1)(a-b)=k(a-b)+a-b\le k(a-b)+a$, in order that $(x, y)\not\in R_k$, we need $k(a-b)+a\le y$.  Then $(x, y)$ can be sent to $(x+b, k(a-b))$ both in $ab$-Yama Nim and $ab$-Triangular Nim, and $(x+b, k(a-b))$ is in $R_k$ because $k(a-b)\le x\le x+b<(k+1)(a-b)+b=k(a-b)+a$.  

When $b<0$ with $ab$-Triangular Nim, again we let $k=\left[\dfrac{x}{a-b}\right]$.  When $x<k(a-b)+a$, then in order that $y\not\in R_k$, $y$ must satisfy the inequality $k(a-b)+a\le y$.  Then because $b<0$, we can move $(x, y)$ to $(x, k(a-b))\in R_k$ in $ab$-Triangular Nim.  On the other hand, if $k(a-b)+a\le x$, then because $x\le y$, we have  $k(a-b)+a\le y$, and also by the choice of $k$, we have $x<(k+1)(a-b)=k(a-b)+a-b$, hence $x+b<k(a-b)+a$.  We can send $(x, y)$ to $(x+b, k(a-b))$, which lies in $R_k$.  

Now we assume that $(x, y)\in R_k$, and we need to show that $(x, y)$ cannot be sent to $\displaystyle \bigcup_{k=0}^{\infty} R_k(a, b)$. Assume not, and $(x, y)$ is sent to $(x', y')\in R_{\ell}$.  As the options of $ab$-Yama Nim is a subset of $ab$-Triangular Nim, it is enough to treat $ab$-Triangular Nim case.  By symmetry, we may assume that $x+b\le x'$ and $y'\le y+a$.  By our assumption $(x, y)\in R_k$ and $(x', y')\in R_{\ell}$, we have $$k(a-b)\le x, y<k(a-b)+a,$$ and $$\ell(a-b)\le x', y' <\ell(a-b)+a.$$  Then we have $$k(a-b)+a>y\ge y'+a\ge \ell(a-b)+a,$$ $$\ell(a-b)+a>x'\ge x+b\ge k(a-b)+b.$$  Combining, we obtain $$k(a-b)+a>\ell (a-b)+a>k(a-b)+b,$$ and subtracting $b$ from all the terms, we have $$(k+1)(a-b)>(\ell+1)(a-b)>k(a-b),$$  and as $a-b>0$, dividing by $a-b$, we can conclude $k+1>\ell+1>k$, which is impossible for integers $k$ and $\ell$.
\end{proof}

In the definition of $ab$-Yama Nim and $ab$-Triangular Nim, we are grateful to Ryuya Hora who pointed out the possibility of considering the case $b\le 0$.

\section{Wythoff twist}

In classical Wythoff Nim, in addition to the usual Nim move, namely taking some tokens from $1$ heap, the player is allowed to take tokens from $2$ heaps, as far as she/he takes the same number of tokens from both heaps.  In this section, we consider similar twists for Yama Nim and Triangular Nim.  

\begin{definition} Let $c\in \Z_{\ge 0}$ be a non-negative integer.  In Yama Nim with $c$-Wythoff twist and Triangular Nim with $c$-Wythoff twist, the set of positions is $X=\Z_{\ge 0}\times \Z_{\ge 0}$, and for a position $(x, y)\in X$, the options are 
\begin{eqnarray*}&&\{(x-i, y+1)\, |\, 2\le i\le x\}\\ &\cup& \{(x+1, y-i)\, |\, 2\le i\le y, \}\\ &\cup& \{(x-i, y-j)\, |\, 1\le i\le x, 1\le j\le y, |i-j|\le c\}\end{eqnarray*}
for Yama Nim with $c$-Wythoff twist, and
\begin{eqnarray*}&&\{(x-i, y+j)\, |\, 2\le i\le x, 1\le j<i\}\\ &\cup& \{(x+j, y-i)\, |\, 2\le i\le y, 1\le j<i\}\\ &\cup& \{(x-i, y-j)\, |\, 1\le i\le x, 1\le j\le y, |i-j|\le c\}\end{eqnarray*} for Triangular Nim with $c$-Wythoff twist.
\end{definition}

\Pps for Yama Nim with $c$-Wythoff twist is easy to describe.

\begin{proposition} The set of \Pps for Yama Nim with $c$-Wythoff twist is $S=S_1 \cup S_2 \cup S_3$ where
\begin{eqnarray*}S_1&=&\{(0,0), (0, 1), (1, 0)\}\\ S_2&=&\{(x, 1)\, |\, x\ge 3+c\}\\ S_3&=&\{(1, y)\, |\, y\ge 3+c\}\end{eqnarray*}
\end{proposition}

\begin{proof} We need to show that (1) when $(x, y)\in S$, and option is outside $S$, and (2) when $(x, y)\not\in S$, it has some option in $S$.

For (1), if $(x, y)\in S_1$, then it has no options at all.  For the rest, by symmetry, we may assume $(x, y)=(x, 1)\in S_2$ with $x\ge 3+c$.  In Yama Nim option, $(x, 1)$ must be sent to $(x-i, 2)$, which is not in $S$.  In $c$-Wythoff option, $(x, 1)$ is sent to $(x-i, 0)$ with $|i-1|\le c$, hence $i\le c+1$ and $x-i\ge x-c-1\ge 2$, which is again not in $S$.

For (2), we may take $(x, y)\not\in S$ with $x\le y$.  When  $x=0$, we have $y\ge 2$, and $(0, y)$ has a Yama Nim option $(1, 0)\in S_1$.  When $x=1$, we have $y\le 2+c$, and $(1, y)$ has a $c$-Wythoff option $(0, y-1-c)\in S_1$.  When $x=2$, in case $y\le 3+c$, then $|(y-1)-2|\le c$ and we  have a $c$-Wythoff option $(2, y)\to (0, 1)\in S_1$.  On the other hand, when $y\ge 4+c$, then we have a $c$-Wythoff option $(2, y)\to (1, y-1)\in S_2$.  Finally when $3\le x \le y$, then we have a Yama Nim option $(x, y)\to (1, y+1)\in S_2$.  \end{proof}

\begin{corollary} In Yama Nim with $c$-Wythoff twist, all positions with Grundy number less than equal to $1$ are in final stage, hence Yama Nim with $c$-Wythoff twist is universally admissible. \end{corollary}

\begin{proof} Whenever $(x, y)\in S$ has some option, it can be sent to $(0, 2)$ or $(2, 0)$, and their options are $\{(1, 0)\}$ and $\{(0, 1)\}$ respectively in $S$, hence $(0, 2)$ and $(2, 0)$ have Grundy number $1$. \end{proof}

\begin{remark} The set of \Pps for Yama Nim with $c$-Wythoff twist is, when $c=0$, $$\{(0,2), (1, 1), (2 ,0), (2, 3), (3, 2)\}\cup \{(x, 3)\, |\, x\ge 6\} \cup \{(3, y)\, |\, y\ge 6\},$$
and for $c\ge 1$, 
$$\{(0,2), (1, 1), (2 ,0)\}\cup \{(x, 2)\, | \, x\ge 5+c\} \cup \{(2, y)\, |\, y\ge 5+c\}.$$
\end{remark}

To state the \Pps for Triangular Nim with $c$-Wythoff twist, it is convenient to define the following sequence.

\begin{definition} When $c\in \Z_{\ge 0}$ is a non-negative integer, we define the sequence $a_{[c], 0}, a_{[c], 1}, a_{[c], 2}, \ldots$ by $$a_{[c], k}=\dfrac{1+c}{2}k^2+\dfrac{1-c}{2}k.$$
In other words, the sequence $a_{[c], 0}, a_{[c], 1}, a_{[c], 2}, \ldots$ are polygonal numbers with the number of sides $c+3$.  When $0\le c \le 3$, \begin{eqnarray*}a_{[0], 0} ,a_{[0], 1}, a_{[0], 2}, \ldots&=&0, 1, 3, 6, 10, 15, 21, \ldots\\ a_{[1], 0} ,a_{[1], 1}, a_{[1], 2}, \ldots&=&0, 1, 4, 9, 16, 25, 36, \ldots\\a_{[2], 0} ,a_{[2], 1}, a_{[2], 2}, \ldots&=&0, 1, 5, 12, 22, 35, 51, \ldots\\a_{[3], 0} ,a_{[3], 1}, a_{[3], 2}, \ldots&=&0, 1, 6, 15, 28, 45, 66, \ldots\end{eqnarray*}
\end{definition}

\begin{theorem} \label{Wythoff Triangular} The set of \Pps for Triangular Nim with $c$-Wythoff twist is $$\{(0,0)\}\cup \{(a_{[c], k}, a_{[c], k+1)})\, |\, k=0,1, 2, 3, \ldots\}\cup  \{(a_{[c], k+1}, a_{[c], k)})\, |\, k=0,1, 2, 3, \ldots\}$$
\end{theorem}  

For example when $c=0$, the list of \Pps $(x, y)$ with $x\le y$ are
$$(0, 0), (0, 1), (1, 3), (3, 6), (6, 10), (10, 15), (15, 21), (21, 28) ,\ldots $$
consecutive triangular numbers, and when $c=1$, the \Pps with $x\le y$ are
$$(0, 0), (0, 1), (1, 4), (4, 9), (9, 16), (16, 25), (25, 36), (36, 49), \ldots$$
and so on.

\begin{proof} We fix $c\in \Z_{\ge 0}$ and simply write $a_k$ for $a_{[c], k}$.  Let $S$ to be $$S=\{(0, 0)\}\cup \{(a_k ,a_{k+1})\, |\, k=0, 1, 2, \ldots\}\cup \{(a_{k+1}, a_k)\, |\, k=0, 1, 2, \ldots\}.$$
We need to show: \begin{enumerate} \item[(1)] For any $(x, y)\in S$, its option $(x', y')$ is not in $S$. \item[(2)] For any  $(x, y)\not\in S$, there is an option  $(x', y')\in S$.\end{enumerate}

For (1), we take $(x, y)\in S$.  As $(0, 0)$ has no options, and by symmetry, we may assume $(x, y)=(a_k, a_{k+1})$.  By definition of $a_k=a_{[c], k}$, we have $a_{k+1}-a_k=k(c+1)+1$, which make an arithmetic progression $1, 1+(c+1), 1+2(c+1), 1+3(c+1), \ldots$, with the initial term $1$ and the common difference $c+1$.  

When $(x', y')$ is a Wythoff option of $(x, y)$, say $x'=x-i$ and $y'=y-j$ with $|i-j|\le c$, then $|x'-y'|=|(x-y)-(i-j)|$, and hence we have $$1+(k-1)(c+1)<1+k(c+1)-c\le |x'-y'|\le 1+k(c+1)+c<1+(k+1)(c+1),$$ then the only possibility for $(x', y')$ to be in $S$ is $\{x', y'\}=\{x, y\}$.  But after the move, the total number of token must decrease, and it is impossible.  

When $(x', y')$ is a Triangular option of $(x, y)$, and assume that $(x, y)\in S$.  As $0<x<y$, and at least in one of the heaps, the number of tokens should increase, so $(x'. y')\not=(0, 0)$, and hence it should be written as $(x', y')=(a_{\ell}, a_m)$ with $|\ell-m|=1$.  When $x'<x$ and $y<y'$, then $\ell<k<k+1<m$, so it cannot satisfy $|\ell-m|=1$.  When $x<x'$ and $y'<y$, then $m\le k<k+1\le \ell$, and the only possibility for $|\ell-m|=1$ is $m=k$ and $\ell=k+1$, but if so, again $x'+y'=x+y$, contradicting to the decreasing condition of the tokens. Therefore, $(x', y')\not\in S$, and (1) holds.

For (2), we take $(x, y)\not\in S$.  If $x=y$, then by Wythoff option, we can send $(x, y)\to (0, 0)$, so  we may assume $x<y$ by symmetry.  Take $k\in \Z_{\ge 0}$ so that $a_k\le x<a_{k+1}$.  If $x+y>a_k+a_{k+1}$, then by Triangular option, we can send $(x, y)\to (a_{k+1}, a_k)$, so we need only to consider the remaining case, namely when  $x+y\le a_k+a_{k+1}$.  In this case, we have $y\le a_k-(x-a_k)$ with $x-a_k\ge 0$ by the choice of $k$, and hence $a_k\le x<y\le a_{k+1}$.  If $a_k=x$ and $a_{k+1}=y$, then $(x, y)=(a_k, a_{k+1})\in S$, contradicting to our choice of $(x, y)\not\in S$, so we have $y-x<a_{k+1}-a_k=1+k(c+1)$.  Take $\ell=\left[\dfrac{y-x-1}{c+1}\right]$, then $1+\ell(c+1)\le y-x < 1+(\ell+1)(c+1)$.  As $y-x-1<k(c+1)$, we have $\ell<k$, and hence $a_{\ell}<a_k\le x$ and $$a_{\ell+1}=a_{\ell}+1+\ell(c+1)\le a_{\ell}+y-x=y-(x-a_{\ell})<y.$$  Then we can send $(x, y)$ to $(a_{\ell}, a_{\ell+1})\in S$ by $c$-Wythoff option, because \begin{eqnarray*}&&|(y-a_{\ell+1})-(x-a_{\ell})|\\&=&|(y-x)-(1+\ell(c+1))|\\&<&1+(\ell+1)(c+1)-(1+\ell(c+1))\\&=&c+1,\end{eqnarray*}  hence $|(y-a_{\ell+1})-(x-a_{\ell})|\le c$.  We are done. \end{proof}

\begin{remark} Let $g\in \Z_{>0}$ be a positive integer and $c\in \Z_{\ge 0}$ a non-negative integer, it seems that there is a sequence $\{a_1, a_2, \ldots\}$ so that if $(x, y)$ has Grundy number $g$ in Triangular Nim with $c$-Wythoff twist, then either $x+y-1=g$ or $(x, y)=(a_k, a_{k+1})$, or $(x, y)=(a_{k+1}, a_k)$ for some $k$.  For example, the list of positions $(x, y)$ with $x\le y$ and Grundy number $1$ in Triangular Nim with $0$-Wythoff twist is $$\{(1, 1), (0, 2), (2, 3) ,(3, 7), (7, 10), (10, 16), (16, 21), (21, 29), (29, 36), (36, 46), \ldots\}$$ so the sequence in this case seems to be  $2, 3, 7, 10, 16, 21, \ldots$, or $a_k=\dfrac{k^2+k+1-(-1)^k}{2}$.  Moreover when $g$ is small compared to $c\ge 1$, then the sequence seems to be $a_k=\dfrac{(1+c)k^2+(1+2g-c)k}{2}$, which seems to work for $c=1$ and $g=1, 2, 3$.  It fails for $g=4$, where the Grundy number of $(5, 8)$ is $4$ for Triangular Nim with $1$-Wythoff twist, where our sequence above predicted $\{(2, 3), (0, 5), (5, 12), (12, 21), (21, 32), \ldots\}$ as the list of Grundy number $4$. For $c=2$, the first case our formula fails is $(27, 37)$ with Grundy number $26$, even though our formula predicts $(27, 57)$ instead of $(27, 37)$ to have Grundy number $26$. When $c=3$, we have found no failing cases so far. \end{remark}

\section{Geometric Progression}

Let $d\in \Z_{\ge 2}$ be an integer  larger than $1$.  We define Triangular Nim with $d$-geometric twist, and Triangular Nim with $d$-sub-geometric twist, as follows:

\begin{definition} The position sets of {\em Triangular Nim with $d$-geometric twist}, and {\em Triangular Nim with $d$-sub-geometric twist} are the same as Triangular Nim, namely $X=\Z_{\ge 0} \times \Z_{\ge 0}$.  In addition to the options of Triangular Nim, in $d$-geometric twist, the player is allowed to take tokens from both heaps, say $i$ from the first heap and $j$ from the second heap, as far as $$\hbox{$i\le 2j-2$ and $j\le 2i-2$}.$$ 

In $d$-sub-geometric twist, in addition, the player is allowed to take $i$ and $j$ tokens as far as $$\hbox{$i\le 2j-1$ and $j\le 2i-1$.}$$\end{definition}

\begin{theorem} The position $(x, y)$ with $x\le y$ is in the set of \Pps for Triangular Nim with $d$-geometric twist if and only if $(x, y)$ lies in $$\left\{\begin{array}{rclc} \{(0, 0), (0, 1), (1, 1)\} & \cup & \{(1, 2), (2, 4), (4, 8), \ldots, (2^k, 2^{k+1}), \ldots\} &\hbox{ \,\, (when $d=2$)}\\
\{(0, 0), (0, 1)\} & \cup & \{(1, d), (d, d^2), (d^2, d^3), \ldots, (d^k, d^{k+1}), \ldots\} &\hbox{ \,\, (when $d>2$)}\end{array}\right.$$\end{theorem}

\begin{proof} Let $S$ be the stated candidate for the set of \Pps, and we need to show that \begin{enumerate}
\item[(1)] If $(x, y)\in S$, then any of its option $(x', y')$ is outside $S$.
\item[(2)] If $(x, y)\not\in S$, then its some option $(x', y')\in S$.
\end{enumerate}

For (1), for $(x, y)=(0, 0)$ or $(0, 1)$, there is no option.  When $d=2$, $(1, 1)$ has no option too, because if $i\le 2j-2$ and $j\le 2i-2$, then $i$ and $j$ must be larger than $1$.  Let us take $k\in \Z_{\ge 0}$ and let $(x, y)=(d^k , d^{k+1})$, and we need to show that no option lies in $S$.  Assume that one of $(0, 0), (0, 1), (1, 0)$ and $(1, 1)$ is one of its options by the $d$-geometric option, say  $(x, y)\to (x-i, y-j)$, then  $i=d^k$ or $d^k-1$ and $j=d^{k+1}$ or $d^{k+1}-1$.  Then we have $$di-2\le d\cdot d^k-2=d^{k+1}-2<d^{k+1}-1\le j$$ hence the condition $di-2\le j$ cannot be satisfied, and this move is not allowed in the geometric option.  In Triangular option, we need to increase at least one heap, and both $x$ and $y$ are at least $1$, we cannot move to $(0, 0), (0, 1), (1, 0)$ or $(1, 1)$.

We need to show that $(x, y)=(d^k, d^{k+1})$ cannot be sent to $(d^{\ell}, d^m)$ with $|\ell-m|=1$.  If $(x, y)$ is sent to $(x', y')$ by Triangular option, then either $x'<x$ and $y<y'$, or $x<x'$ and $y'<y$.  In the case $x'<x$ and $y<y'$, $\ell<k<k+1<m$, hence $(x', y')\not\in S$.  In the case $x<x'$ and $y'<y$, we have $k<\ell$ and $m<k+1$, so the only possibility for $|\ell-m|=1$ i $\ell=k+1$ and  $m=k$, but then it contradicts to our rule $x'+y'<x+y$, so Triangular option cannot send $(x, y)=(d^k, d^{k+1})$ to a position in $S$.  Assume that we can send $(x, y)=(d^k, d^{k+1})$ to $(x', y')=(x-i, y-j)=(d^{\ell}, d^{\ell+1})$, then $\ell<k$, and $i=d^k-d^{\ell}$ and $j=d^{k+1}-d^{\ell+1}$, and then $j=di>di-2$, so the condition $j\le di-2$ is not satisfies.  Also if we try to send $(x, y)=(d^k, d^{k+1})$ to $(x', y')=(x-i, y-j)=(d^{\ell+1}, d^{\ell})$, then $j=d^{k+1}-d^{\ell}>d^{k+1}-d^{\ell+2}=di>di-2$, again the condition $j<di-2$ fails.  Therefore, positions in $S$ cannot be sent to a position in $S$.

Now we show (2).   We take $(x, y)\not\in S$ with $x\le y$.  If $x=y$, then $d$-geometric option allows us to send $(x, y)\to (0, 0)\in S$, so we may assume $x<y$.  If $x=0$, then $2\le y$, and hence Triangular option allows us to send $(x, y)$ to $(1, 0)\in S$, so we may assume $x\ge 1$, and we can take $k\in \Z_{\ge 0}$ so that $d^k\le x<d^{k+1}$.  If $x+y>d^k+d^{k+1}$, then by Triangular option, we can send $(x, y) \to  (d^{k+1}, d^k)$, so we may assume $x+y\le d^k+d^{k+1}$, or $y\le d^{k+1}-(x-d^k)\le d^{k+1}$.  We write $\alpha=x-d^k\ge 0$ and $\beta=d^{k+1}-y\ge 0$.  Moreover if $\alpha=\beta=0$, then $(x, y)=(d^k, d^{k+1})$, which contradicts to our assumption that $(x, y)\not\in S$, so either $\alpha>0$ or $\beta>0$, and in particular, $y-x<d^{k+1}-d^k$.  

Now we show that we may assume $y-x\ge d-1$.  First, if $d=2$ and $x=1$, then $3\le y$ and $(x, y)$ can be sent to $(2, 1)$ by Triangular option.  Also if either $d>2$ or $x\ge 2$, then we show that $(x, y)$ can be sent to $(0, 1)$ by $d$-geometric option.  In this case, $i=x, j=y-1\le x+d-2$, and $dj-2-i\ge (d-1)x+d^2-2d-2$, which must be larger than $0$, by the condition $d>2$ or $x\ge 2$.  

As $y-x\ge d-1$, we can choose $\ell\in \Z_{\ge 0}$ so that $$d^{\ell}(d-1)\le y-x<d^{\ell+1}(d-1).$$  In fact we may take $$\ell=\left[\log_d\dfrac{y-x}{d-1}\right],$$ which is non-negative because $y-x\ge d-1$.  As $d^{\ell}(d-1)\le y-x<d^k(d-1)$, so $\ell<k$, and $d^{\ell}<x, d^{\ell+1}<y$.  Now $d$-geometric option allow us to send $(x, y)\to (d^{\ell}, d^{\ell+1})$ for most of the cases, because by setting $(x-i, y-j)=(d^{\ell}, d^{\ell+1})$, we have $i=d^k+\alpha-d^{\ell}, j=d^{k+1}-\beta-d^{\ell+1}$, and as $d^{\ell}(d-1)\le y-x=d^{k+1}-d^k-(\alpha+\beta)$, we obtain $\alpha+\beta\le (d-1)(d^k-d^{\ell})$, and hence 
\begin{eqnarray*} dj-2-i&=&d^{k+2}-d\beta-d^{\ell+2}-d^k-\alpha+d^{\ell} \\
&\ge & d^{k+2}-d^{\ell+2}-d^k+d^{\ell}-d(\beta+\alpha)-2\\
&\ge&d^{k+2}-d^{\ell+2}-d^k+d^{\ell}-d(d-1)(d^k-d^{\ell})-2\\
&=&(d-1)(d^k-d^{\ell})-2\\
&\ge&-1,\end{eqnarray*} 

where the inequality of the third line follows from  $\alpha+\beta\le (d-1)(d^k-d^{\ell})$.  Notice that in the inequality in the last line, the equality holds if and only if $d=2, k=1$ and $\ell=1$, which happens only when $(x, y)=(2, 3)$, but then it can be sent to $(0, 1)$ by $2$-geometric option, and otherwise $i\le dj-2$.

Also for the other inequality, we have
\begin{eqnarray*} di-2-j&\ge& d(d^k+\alpha-d^{\ell})-2-(d^{k+1}-\beta)\\
&=&d\alpha+\beta-2.\end{eqnarray*} If $\alpha>0$, then the right hand side is $d\alpha+\beta-2\ge d-2\ge 0$, so $j\le di-2$.  So the only possibility for $j>di-2$ is $\alpha=0$ and $\beta=1$, so $(x, y)=(d^k, d^{k+1}-1)$, but then $d$-geometric option can send $(d^k, d^{k+1}-1)\to (0, 1)$, because in this case, $i=d^k$ and $j=d^{k+1}-2=di-2$, and $dj-2-i=d^k(d^2-1)-2d$, which is larger than $0$ if $k>0$, and if $k=0$, $dj-2-i=d^2-2d-1$, which is negative only for $d=2$, but then $(x, y)=(2^0, 2^1-1)=(1, 1)\in S$, contradicting to our assumption $(x, y)\not\in S$.
\end{proof}

\begin{theorem}
The position $(x, y)$ with $x\le y$ is in the set of \Pps for Triangular Nim with $d$-sub-geometric twist if and only if $(x, y)$ is in $$\{(0, 0), (0, 1)\}\cup \{(1, \dfrac{d^2-1}{d-1}), (\dfrac{d^2-1}{d-1}, \dfrac{d^3-1}{d-1}), \ldots, (\dfrac{d^k-1}{d-1}, \dfrac{d^{k+1}-1}{d-1}), \ldots\}.$$
\end{theorem}

In particular, when $d=2$, the \Pps with $x\le y$ looks like $$(0, 0), (0, 1),  (1, 3), (3, 7), (7, 15), (15, 31), (31, 63), \ldots$$
which is described by Mersenne numbers.

\begin{proof}  Again we let $S$ be the stated candidate for the set of \Pps, and we need to show that \begin{enumerate}
\item[(1)] If $(x, y)\in S$, then any of its option $(x', y')$ is outside $S$.
\item[(2)] If $(x, y)\not\in S$, then its some option $(x', y')\in S$.
\end{enumerate}

For (1), there are no options from $\{(0, 0), (0, 1)\}$.  Let $a_k=\dfrac{d^k-1}{d-1}, k=1, 2, 3, \ldots$, then the argument that Triangular options of $S$ is out of $S$ is now standard in this paper, and is omitted.  For $d$-sub-geometric options, if $(a_k , a_{k+1})$ is sent to $(a_k-i, a_{k+1}-j)=(a_{\ell}, a_{\ell+1})$, then $k>\ell$ and $j=d \dfrac{d^k-d^{\ell}}{d-1}=di$, this move is not allowed.  If $(a_k, a_{k+1})$ is sent to $(a_k-i, a_{k+1}-j)=(0, 0), (0, 1)$ or $(1, 0)$, then $i\le1+d+d^2+\cdots+c^{k-1}$, $j\ge d+d^2+\cdots+c^k=d(1+d+\cdots+d^{k-1})\ge di>di-1$, so again this move is not allowed.  

For (2), if $x=0$, then $y\ge 2$ and we have Triangular option $(0, y)\to (1, 0)$, so we may assume that $x>0$.  Also if $x=y$, then $(x, y)$ has a $d$-sub-geometric option $(0, 0)$, so we may assume $x<y$.  For such $(x, y)\not\in S$, take $k\in \Z_{\ge 0}$ so that $a_k\le x<a_{k+1}$.  If $x+y>a_k+a_{k+1}$, then Triangular option can send $(x, y)\to (a_{k+1}, a_k)$ so we may assume $y\le a_{k+1}-(x-a_k)\le a_{k+1}$, and we set $\alpha=x-a_k\ge 0$ and $\beta=a_{k+1}-y$.  We have $\alpha\ge 0, \beta\ge 0$, and if $\alpha=\beta=0$, it contradicts to our assumption that $(x, y)\not\in S$, so $\alpha+\beta>0$, and $y-x<a_{k+1}-a_k$.  

When $y-x<d$, then we show that $(x, y)$ can be sent to $(x-i, y-j)=(0, 1)$ by $d$-sub-geometric option.  In this case, we have $i=x, j=y-1$, and we need to check $i\le dj-1$ and $j\le di-1$.   As $x<y$, $dj-1-i\ge d(x+1)-1-x=(d-1)(x+1)\ge 2$.  Also because $y\le x+d-1$, we have $di-j\ge dx-x-d+1=(d-1)(x-1)\ge 0$.  So We may assume $y-x\ge d=a_2-a_1$.  We take $\ell$ so that $a_{\ell+1}-a_{\ell}\le y-x<a_{\ell+2}-a_{\ell+1}$, or equivalently,
$d^{\ell}\le y-x<d^{\ell+1}$.
As $a_{\ell+1}-a_{\ell}=d^{\ell}\le y-x < a_{k+1}-a_k=d^k$, the inequality $\ell<k$ follows, hence $y>x\ge a_k\ge a_{\ell+1}>a_{\ell}$.  We show that $(x, y)$ can be sent to $(x-i, y-j)=(a_{\ell}, a_{\ell+1})$ by the $d$-sub-geometric option.  In this case, $$i=a_k+\alpha-a_{\ell}=d^{\ell}+d^{\ell+1} +\cdots+d^{k-1}+\alpha$$
and $$j=a_{k+1}-\beta-a_{\ell+1}=d^{\ell+1}+d^{\ell+2}+\cdots+d^k-\beta.$$
Then $di-j=d\alpha+\beta\ge 1$ because at least one of $\alpha$ and $\beta$ is positive, hence $j\le di-1$.  

As $y-x=d^k-(\alpha+\beta)\ge a_{\ell+1}-a_{\ell}=d^{\ell}$, we have $\alpha+\beta\le d^k-d^{\ell}$, and hence \begin{eqnarray*} dj-i&=&d^{\ell+2}+d^{\ell+3}+\cdots+d^{k+1}-d\beta-d^{\ell}-\cdots-d^{k-1}-\alpha\\
&\ge& d^{k+1}+d^k-d^{\ell+1}-d^{\ell}-d(\alpha+\beta)\\
&\ge & d^k-d^{\ell}>0,\end{eqnarray*} therefore we have $i\le dj-1$, and our option $(x, y)\to (a_{\ell}, a_{\ell+1})$ is justified.  \end{proof}

\begin{question} Given a sequence $a_1, a_2, a_3, a_4, \ldots, $ then can you cook up some natural ruleset so that \Pps $(x, y)$ with $x\le y$ are $\{(a_1, a_2), (a_2, a_3), (a_3, a_4), \ldots\}$, possibly finitely many exceptions?   What conditions do we need for the sequence?   For example, can you find a natural ruleset for Fibonacci sequence?  \end{question}

\section{$3$ heaps}

There are several ways to generalize Triangular Nim to $3$ heaps.  Some rulesets give trivial \Pps, but some are not.

\begin{definition}
Let $X=\Z_{\ge 0}\times \Z_{\ge 0}\times \Z_{\ge 0}$ be the set of positions, namely we have $3$ heaps of tokens.  
\begin{enumerate}
\item[(1)] The player takes more than $1$ token from one heap, choose another heap, and returns at least $1$ token to that heap, so that the total number of tokens in the heaps decreases.
\item[(2)] The player takes more than $1$ token from one heap, and returns tokens to the other $2$ heaps, both of them at least $1$, so that the total number of tokens in the heaps decreases.
\item[(3)] The player takes more than $1$ token from one  heap, and returns tokens to the other $2$ heaps, at least $1$ in total, but she/he may return only to one heap, so that the total number of tokens in the heaps decreases.
\item[(4)] The player choose $2$ heaps, take at least $1$ token from both of them, and return at least $1$ token to the other heap, so that the total number of tokens in the heaps decreases.
\item[(5)]The player takes more than $1$ token from one heap, choose another heap, and returns at least $1$ token to that heap, so that the total number of tokens in the heaps decreases.  But once a heap loses all its tokens, then that heap is removed.
\item[(6)] The player takes more than $1$ token from one heap, and returns tokens to the other $2$ heaps, both of them at least $1$, so that the total number of tokens in the heaps decreases.  But once a heap loses all its tokens, then that heap is removed.
\end{enumerate}\end{definition}

The \Pps are, for $(x, y, z)$ with $x\le y\le z$, as follows: 

For the ruleset (1), $$\{(0,0,1)\}\cup \{(0,a, a)\, | \, a\in \Z_{\ge 0}\}\cup \{(a, a, a)\, |\, a\in \Z_{\ge 0}\}\cup \{(a, a, 2a)\, |\, a\in \Z_{\ge 0}\}.$$

For the ruleset (2), $$\{(x, y, z)\, |\, z\le 2\}\cup \{(0, y, z)\, |\, z\le y+1\} \cup \{(1, y, y)\, |\, y\in \Z_{\ge 1}\}.$$

For the ruleset (3), $$\{(x, y, z)\, |\, z\le 1\} \cup \{(1, 1, 2)\} \cup \{(0, y, y)\, |\, y\in \Z_{\ge 0}\}.$$

For the ruleset (4), $$\{(0, 0, z)\, |\, z\in \Z_{\ge 0}\}.$$

For the ruleset (5), $$\{(x, y, z)\, |\, z\le 1\} \cup\{(0, 0, z)\, |\, z\in\Z_{\ge 0}\}\cup \{(1, y, y)\, |\, y\in \Z_{\ge 1}\}\cup\{(1, 1, 2)\}\cup \{(x, x, 2x-1)\, |\, x\in \Z_{\ge 1}\}.$$

 For the ruleset (6), $$\{(x, y, z)\, |\, z\le 2\} \cup \{(2, 2, 3)\} \cup \{(0, 0, z)\, |\, z\in \Z_{\ge 0}\} \cup \{(1, y ,y) \, |\, y\in \Z_{\ge 1}\}$$
 
 \section{Transfinite Nim}
 
 Fix some infinite set $Y$, and consider all the ordinal numbers whose cardinality is less than or equal to $|Y|$, and we work in the set of these ordinal numbers, which we will write as $\OO$.  Cantor normal form allows us to write an ordinal $\alpha\in \OO$ as $$\alpha=A_1\omega^{a_1}+A_2\omega^{a_2}+\cdots+A_r\omega^{a_r}$$ uniquely, where $r\in \Z_{\ge 0}$ is a non-negative integer, $A_i\in \Z_{>0}, (i=1, 2, \ldots, r)$ are positive integers, $\omega=\{0, 1, 2, 3, \ldots\}$ is the smallest infinite ordinal number, $a_1>a_2>\cdots>a_r$ are strictly decreasing ordinal numbers \cite[Thm 2, page 320]{Si}.  
 
 Ordinal numbers can be regarded as positions in Conway's surreal numbers \cite{C}, and hence we can freely add them as if $\omega$ is a free variable, and this semi-group is embedded in the group of surreal numbers.  This sum of ordinal numbers agree with natural sums \cite[p363]{Si} defined by Hessenberg.  Transfinite Nim games are treated in \cite[p330]{BCG}.  Using this notion of natural sum, we can play transfinite Wythoff Nim, transfinite Yama Nim, and transfinite Triangular Nim.
 
 \begin{definition} Let $Y=\OO\times \OO$ be the set of positions.
In transfinite Wythoff Nim, the player can move from $(\alpha, \beta)$ to either $(\alpha', \beta)$ with $\alpha'<\alpha$, $(\alpha, \beta')$ with $\beta'<\beta$ or to $(\alpha', \beta')$ with $\alpha+\beta'=\beta+\alpha'$.  

In transfinite Yama Nim, the player can move from $(\alpha, \beta)$ to either $(\alpha', \beta+1)$ with $\alpha'+1<\alpha$ or $(\alpha+1, \beta')$ with $\beta'+1<\beta$. 

In transfinite Triangular Nim, the player can move from $(\alpha, \beta)$ to $(\alpha', \beta')$, either 
\begin{enumerate}
\item[(i)] $\alpha'<\alpha$ and $\beta<\beta'$, or
\item[(ii)] $\alpha'>\alpha$ and $\beta>\beta'$
\end{enumerate}
in such a way that $\alpha+\beta>\alpha'+\beta'$.
\end{definition}

As $\OO$ is a well-ordered set, any strictly decreasing sequence terminates in finitely many steps, and in all definitions above, the total (in the sense of natural sum) number of tokens in the heaps decreases strictly, so even though they may have infinitely many tokens, the games always end in finitely many steps, and we can decide the winner.

\begin{theorem} In transfinite Wythoff game, $(\alpha, \beta)$ with $\alpha=A_1\omega^{a_1}+A_2\omega^{a_2}+\cdots+A_r\omega^{a_r}$ and $\beta=B_1\omega^{b_1}+B_2\omega^{b_2}+\cdots+B_s\omega^{a_s}$, is in \Pps if and only if $r=s$, $a_1=b_1, a_2=b_2, \ldots, a_r=b_r$, and $(A_i, B_i)\in \Z_{>0}\times \Z_{>0}$ are in \Pps of finite Wythoff Nim for all $i=1, 2, \cdots, r$.
In transfinite Yama Nim and transfinite Triangular Nim, $(\alpha, \beta)$ is in \Pp if and only if either $\alpha+1=\beta$ or $\beta+1=\alpha$.
\end{theorem}

We also studied Grundy numbers of transfinite Triangular Nim, which will appear elsewhere.

\end{document}